\documentclass[12pt,a4paper]{amsart}

\usepackage[centertags]{amsmath}
\usepackage{amsfonts,ifthen,latexsym,amssymb,amsthm, amscd}

\usepackage[T1]{fontenc}
\usepackage[utf8]{inputenc}

\usepackage[usenames,dvipsnames]{color}

\definecolor{labelkey}{gray}{.20}
\definecolor{refkey}{gray}{.20}
\definecolor{eqkey}{gray}{.20}

\allowdisplaybreaks

\usepackage[american]{babel} 

\usepackage[perpage]{footmisc}

\usepackage{graphicx}

\usepackage{url}
\usepackage{hyperref}
\hypersetup{colorlinks=true,citecolor=blue,filecolor=blue,linkcolor=blue,urlcolor=blue}

\usepackage{version} 

\def\mathclap#1{\text{\hbox to 0pt{\hss$\mathsurround=0pt#1$\hss}}}

\newcommand{\beq}{\begin{equation}}
\newcommand{\eeq}{\end{equation}}
\newtheorem{theorem}{Theorem}
\newtheorem*{theorem*}{Theorem}
\newtheorem{lemma}[theorem]{Lemma}
\newtheorem*{lemma*}{Lemma}
\newtheorem{proposition}[theorem]{Proposition}

\newtheorem{problem}[theorem]{Problem}
\newtheorem{conjecture}[theorem]{Conjecture}

\newtheorem{cory}[theorem]{Corollary}

\theoremstyle{definition}
\newtheorem{remark}[theorem]{Remark}
\newtheorem{remarks}[theorem]{Remarks}

\newcommand{\al}{{\alpha}}

\newcommand{\be}{{\beta}}

\newcommand{\eps}{{\varepsilon}}

\newcommand{\ga}{{\gamma}}
\newcommand{\Ga}{{\Gamma}}

\newcommand{\ka}{{\kappa}}
\newcommand{\la}{{\lambda}}

\renewcommand{\phi}{{\varphi}}

\newcommand\aut{\operatorname{Aut}}

\newcommand\R{\mathbb R}
\newcommand\Q{\mathbb Q}
\newcommand\Z{\mathbb Z}
\newcommand\C{\mathbb C}
\newcommand\N{\mathbb N}

\newcommand\ZZ{\mathbf Z}

\newcommand{\into}{\hookrightarrow}
\newcommand{\actson}{\curvearrowright}

\newcommand{\wh}[1]{{\widehat {#1}}}
\newcommand{\cal}[1]{{\mathcal #1}}

\makeatletter
\newcommand*\if@single[3]{%
  \setbox0\hbox{${\mathaccent"0362{#1}}^H$}%
  \setbox2\hbox{${\mathaccent"0362{\kern0pt#1}}^H$}%
  \ifdim\ht0=\ht2 #3\else #2\fi
  }
\newcommand*\rel@kern[1]{\kern#1\dimexpr\macc@kerna}
\newcommand*\widebar[1]{\@ifnextchar^{{\wide@bar{#1}{0}}}{\wide@bar{#1}{1}}}
\newcommand*\wide@bar[2]{\if@single{#1}{\wide@bar@{#1}{#2}{1}}{\wide@bar@{#1}{#2}{2}}}
\newcommand*\wide@bar@[3]{%
  \begingroup
  \def\mathaccent##1##2{%
    \if#32 \let\macc@nucleus\first@char \fi
    \setbox\z@\hbox{$\macc@style{\macc@nucleus}_{}$}%
    \setbox\tw@\hbox{$\macc@style{\macc@nucleus}{}_{}$}%
    \dimen@\wd\tw@
    \advance\dimen@-\wd\z@
    \divide\dimen@ 3
    \@tempdima\wd\tw@
    \advance\@tempdima-\scriptspace
    \divide\@tempdima 10
    \advance\dimen@-\@tempdima
    \ifdim\dimen@>\z@ \dimen@0pt\fi
    \rel@kern{0.6}\kern-\dimen@
    \if#31
      \overline{\rel@kern{-0.6}\kern\dimen@\macc@nucleus\rel@kern{0.4}\kern\dimen@}%
      \advance\dimen@0.4\dimexpr\macc@kerna
      \let\final@kern#2%
      \ifdim\dimen@<\z@ \let\final@kern1\fi
      \if\final@kern1 \kern-\dimen@\fi
    \else
      \overline{\rel@kern{-0.6}\kern\dimen@#1}%
    \fi
  }%
  \macc@depth\@ne
  \let\math@bgroup\@empty \let\math@egroup\macc@set@skewchar
  \mathsurround\z@ \frozen@everymath{\mathgroup\macc@group\relax}%
  \macc@set@skewchar\relax
  \let\mathaccentV\macc@nested@a
  \if#31
    \macc@nested@a\relax111{#1}%
  \else
    \def\gobble@till@marker##1\endmarker{}%
    \futurelet\first@char\gobble@till@marker#1\endmarker
    \ifcat\noexpand\first@char A\else
      \def\first@char{}%
    \fi
    \macc@nested@a\relax111{\first@char}%
  \fi
  \endgroup
}
\makeatother

\newcommand*\mcapinn[2]{\vcenter{\hbox{$\mathsurround=0pt
  \ifx\displaystyle#1\textstyle\else#1\fi\bigcap$}}}

\newcommand*\mcupinn[2]{\vcenter{\hbox{$\mathsurround=0pt
  \ifx\displaystyle#1\textstyle\else#1\fi\bigcup$}}}

\newcommand\scdot{{\cdot}}

\newcommand\diag{\operatorname{Diag}}
\newcommand\dimvn{\dim_\text{vN}}

\newcommand{\cc}{\cal}
\newcommand{\Zmod}[1]{\ZZ_{#1}}
\newcommand{\0}{{\texttt{\large 0}}}
\newcommand{\uz}{{\underline{\texttt{\large 0}}}}
\newcommand{\1}{{\texttt{\large 1}}}
\newcommand{\uo}{{\underline{\texttt{\large 1}}}}

\renewcommand{\ge}{\geqslant}
\usepackage{tikz}
\usetikzlibrary{arrows}
\renewcommand{\ge}{\geqslant}
\renewcommand{\le}{\leqslant}

\begin{document}

\title[Irrational $l^2$-invariants]{Irrational $l^2$-invariants arising from the lamplighter group}
\author{{\L}ukasz Grabowski}
\let\thefootnote\relax\footnote{\hspace{-18pt}\textit{Email:} \texttt{graboluk@gmail.com}}
\let\thefootnote\relax\footnote{\hspace{-18pt}2010 \textit{Mathematics Subject Classification:} 20C07, 20F65, 57M10.}
\let\thefootnote\relax\footnote{\hspace{-18pt}\textit{Key words and phrases:} $l^2$-invariants, Atiyah conjecture, Novikov-Shubin invariants, $l^2$-Betti numbers.}
\maketitle
\vspace{-30pt}
\begin{center}
\textit{Mathematics Institute, University of Warwick, Coventry, CV4 7AL, UK}\end{center}
\begin{abstract}
We show that the Novikov-Shubin invariant of an element of the integral group ring of the lamplighter group $\ZZ_2\wr\ZZ$ can be irrational. This disproves a conjecture of Lott and L\"uck. Furthermore we show that every positive real number is equal to the Novikov-Shubin invariant of some element of the real group ring of $\ZZ_2\wr \ZZ$. Finally we show that the $l^2$-Betti number of a matrix over the integral group ring of the group $\ZZ_p\wr \ZZ$, where $p$ is a natural number greater than $1$, can be irrational. As such the groups $\ZZ_p\wr\ZZ$ become the simplest known examples which give rise to irrational $l^2$-Betti numbers.


\end{abstract}

\maketitle


\vspace{5pt}
Let $\Ga$ be a countable discrete group. A real number $r$ is said to be an
\textit{$l^2$-Betti number arising from $\Ga$} if there is a matrix $T$ with entries in the integral group ring $\Z [\Ga]$, such that the \textit{von Neumann dimension} of the kernel of $T$ is equal to $r$. 

The motivation for the name is as follows: when  $r$ is an $l^2$-Betti number arising from $\Ga$, then there exists a normal covering $M$ of a finite CW-complex whose deck transformation group is $\Ga$, and such that one of the $l^2$-Betti numbers of $M$ is equal to $r$.  We refer to the very readable introduction \cite{Eckmann_intro} for more details.

The following problem is a fine-grained version of a question asked by Atiyah in \cite{Atiyah1976}.

\begin{problem}[The Atiyah problem for $\Ga$] What is the set of $l^2$-Betti numbers arising from $\Ga$?
\end{problem}
Let us denote this set by $\cal C(G)$. For a class of groups
$C$ define $\cc C(C) = \cup_{\Ga\in C} \cc C(\Ga)$. 

So far $\cc C(\Ga)$ has been computed only in cases where $\cc C(\Ga)$ turns out to be a subset of $\Q$. In fact, the statement  known as the \textit{Atiyah conjecture for torsion-free groups} says that $\cc C(\Ga)=\N$ for any
torsion-free group, and before \cite{Dicks_Schick} it was widely conjectured that $\cc C(\Ga) \subset \Q$ for every group $\Ga$. However, \cite{Dicks_Schick} gives an example of a group ring element  $T$ together with an heuristic argument showing why the von Neumann dimension of $\ker T$ is probably irrational. That example is based on \cite{Grigorchuk_Zuk2001}, where a weaker form of the Atiyah conjecture was disproved.

Only recently Austin \cite{arxiv:austin-2009} obtained  a definite result by proving 
that $\cc C($Finitely generated groups$)$ is uncountable.  His results were extended and simplified in  \cite{grabowski-on-turing-dynamical-systems-and-the-atiyah-problem}  and \cite{arxiv:pichot_schick_zuk-2010}, and additional examples were found in  \cite{arxiv:lehner_wagner-2010}.

All the groups $G$ for which it was shown $\cal C(G)\not\subset \Q$ have one of the lamplighter groups $\ZZ_p\wr \ZZ$, where $p$ is a natural number greater than $1$, as a subgroup, but are substantially more complicated than that.


Our first result is as follows.

\begin{theorem}\label{thm-intro-zero} There is a matrix $T$ with entries in the group ring $\Z[\ZZ_p\wr\ZZ]$ such that 
$$
\dimvn\ker T = 1344\left( \frac{4p^3+3p^2+2p-1}{8p^3} + \frac1{8p^3} \sum_{k=1}^\infty \left(\frac{p-1}{p}\right)^{k+2^{k}}\right),
$$
which is a transcendental number.
\end{theorem}

In the view of the preceding discussion, the following problem captures the limit of the currently available methods for finding groups $\Ga$ such that $\cal C(\Ga) \not\subset \Q$.

\begin{problem} Does $\cc C(\Ga) \not\subset \Q$ imply $\Zmod{p}\wr \ZZ\subset \Ga$ for some $p$?
\end{problem}

As mentioned above, $\cc C(\Ga)$ has been computed only in the cases where in fact $\cc C(\Ga)\subset \Q$. Since $\ZZ_p\wr\ZZ$ are  the simplest groups for which we know $\cal C(\Ga)\not\subset\Q$, it is natural to ask the following.

\begin{problem} Is there a description of $\cc C(\Zmod{2}\wr \ZZ)$ substantially different from the definition?
\end{problem}

To state a more concrete problem:  is $\sqrt 2 \in \cc C(\Zmod{2}\wr \ZZ)$ ?

\vspace{5pt}
For our second result let us recall the definition of another spectral invariant associated to an element of a group ring, the \textit{Novikov-Shubin invariant}. It measures the growth of the number of eigenvalues around $0$. More precisely, given a self-adjoint $T\in \C[\Ga]$, the Novikov-Shubin invariant of $T$ is defined as 
\beq\label{eq-ns-def}
\al(T) := \liminf_{\la\to 0^+}\frac{\log(\mu_T((0,\la]))}{\log(\la)},
\eeq
where $\mu_T$ is the spectral measure of $T$ (see \cite[Chapter 2]{Lueck:Big_book} for more details).

\begin{remarks} (i) It is irrelevant whether we take $\mu_T((0,\la])$ or $\mu_T((0,\la))$ in \eqref{eq-ns-def}. However, it is important that we do not include $0$, since otherwise $\al(T)$ would be equal to $0$ whenever the spectral measure of $T$ has an atom at $0$. It is also irrelevant what is the base of the logarithm. It is convenient for us to take the base-$2$ logarithm.

(ii) Both the numerator and the denominator are negative when $\la$ is sufficiently small, so $\al(T) \in [0,\infty]$.

(iii) If for some $d$ and all $\eps$ there is a constant $C>0$ such that for sufficiently small $\la$ we have $\frac1C\la^{d+\eps} < \mu_T((0,\la)) < C\la^{d-\eps}$ then a short computation shows that $\al(T) = d$.
\end{remarks}

Lott and L\"uck \cite{lott-lueck-l2-topological-invariants-of-3-manifolds} proposed the following conjecture.

\begin{conjecture}\label{conj-lott-lueck-intro} When $T\in \Z[\Ga]$ then $\al(T) >0$ and $\al(T)\in \Q$.
\end{conjecture}

For partial results and the motivations for Conjecture \ref{conj-lott-lueck-intro} see \cite[Section 2.5]{Lueck:Big_book}. For counterexamples to the positivity part see \cite{2014arXiv1409.3212G}. In the present paper we construct $T\in \Z[\ZZ_2\wr \ZZ]$ such that $\al(T)\notin \Q$. In fact we show the following.

\begin{theorem}\label{thm-intro-ns}
There is a family $T(b)\in \R[\ZZ_2\wr\ZZ]$, $b\in (1,\infty)$ such that for $b\in \Q$ we have $T(b)\in \Q[\ZZ_2\wr\ZZ]$ and $\al(T(b)) = \frac1{2\log_2(b)}$.
\end{theorem}

Note that the Novikov-Shubin invariant of $T$ and $kT$ is the same for $k>0$, and so we also obtain examples of $T\in \Z[\ZZ_2\wr\ZZ]$ with  irrational Novikov-Shubin invariants.

To the author's best knowledge, the counterexamples to the rationality part of Conjecture \ref{conj-lott-lueck-intro} were not known before even if $\Z[\Ga]$ is replaced by $\R[\Ga]$. The family $T(b)$ is a modification of the operator studied by Grigorchuk and \.Zuk \cite{Grigorchuk_Zuk2001}.

As in the case of $l^2$-Betti numbers, when $r$ is a Novikov-Shubin invariant of some $T\in \Z[G]$,  then there exists a normal covering $M$ of a finite CW-complex whose deck transformation group is $\Ga$, and such that one of the Novikov-Shubin invariants of $M$ is equal to $r$. Conjecture \ref{conj-lott-lueck-intro} could still be true in the case of a finite \textit{aspherical} CW-complex.

Thereom \ref{thm-intro-ns} has an interesting consequence that the set of the Novikov-Shubin invariants of all the elements of $\Q[\ZZ_2\wr\ZZ]$, which  is countable, is  different than the set of the Novikov-Shubin invariants of all the elements of $\R[\ZZ_2\wr\ZZ]$. The analogous question has been asked among the experts for $l^2$-Betti numbers, since there are classes of torsion-free groups for which the Atiyah conjecture is known for $\Q[\Ga]$ but not for $\R[\Ga]$.

\begin{problem}\label{problem-9}
Is it the case that for every $T\in \R[\Ga]$ there exists $T'\in \Q[\Ga]$ such that $\dimvn\ker T = \dimvn\ker T'$?
\end{problem}

Although Theorem \ref{thm-intro-ns} shows that the answer is negative when we replace $\dimvn\ker T= \dimvn\ker T'$ with $\al(T) = \al(T')$, the author believes that at least for $\Ga= \ZZ_2 \wr \ZZ$ the answer to Problem \ref{problem-9} is positive.


\vspace{5pt} The structure of the article is as follows. In the next section we describe the computational tool, in a generality which is just enough for the proof of Theorem \ref{thm-intro-ns}. A general version is presented in \cite[Section 2]{grabowski-on-turing-dynamical-systems-and-the-atiyah-problem} and we refer there for the proofs. Various variants of it were also used for example in \cite{MR1436310}, \cite{Dicks_Schick}, \cite{Lehner_Neuhauser_Woess:On_the_spectrum_of},  \cite{arxiv:austin-2009}, and \cite{arxiv:pichot_schick_zuk-2010}.

In Section \ref{sec-ns} we prove Theorem \ref{thm-intro-ns}. Section \ref{sec-tool-non-basic} presents a slightly different version of the computational tool, which is then used in Section \ref{sec-the-operator} to prove Theorem \ref{thm-intro-zero}. 

Some elementary linear algebra computations are deferred to the appendix.

\vspace{5pt}\textbf{Notation.} The rings of integer, rational, real and complex numbers are $\Z$, $\Q$, $\R$ and $\C$. The cyclic group of order $p$ is $\ZZ_p$ and the infinite cyclic group is $\ZZ$. We fix a generator of $\ZZ$ and denote it by $t$. Given an action $\Ga\actson X$, the result of the action of $\ga\in \Ga$ on $x\in X$ is denoted by $\ga.x$. For example the  translation action of $\ZZ\actson \Z$ is, by definition, given by $t.k:=k{+}1$.

Given two groups $A$ and $B$ the wreath product $A\wr B$ is defined to be $B\ltimes \oplus_B A$, where the action $B\actson \oplus_B A$ is by shifting the coordinates from the left. However, in the case $B=\ZZ$, we write $\ZZ\ltimes \oplus_\Z A$ because it is easier to refer to the coordinates of an element of $\oplus_\Z A$  (which are simply integer numbers) than to the coordinates of an element of $\oplus_\ZZ A$ (which are powers of $t$).

The neutral element of a group is denoted by $e$.

\vspace{5pt}\textbf{Note on chronology.} The first version of this article submitted to arXiv in 2010 contained only Theorem \ref{thm-intro-zero}. Theorem \ref{thm-intro-ns} was added in 2014.

\vspace{5pt}\textbf{Thanks.} The author thanks an anonymous referee for useful comments and especially for pointing out a gap in the proof of Lemma \ref{lem-matrix-lemma}. The author thanks also {\'S}wiatos{\l}aw Gal,  Holger Kammeyer, Jarek K\k{e}dra, Manuel Koehler, Thomas Schick and Andreas Thom for useful comments. 

The author was supported by EPSRC at Imperial College London and Oxford University, by EPSRC grant EP/K012045/1 at University of Warwick, by Austrian Science Foundation project P25510-N26 during author's stay at T.U. Graz, and by Fondations Sciences Math{\'e}matiques de Paris during the program \textit{Marches Al\'eatoires et G\'eom\'etrie Asymptotique des Groupes} at Institut Henri-Poincar\'e.

\renewcommand{\0}{\,{\texttt{\large 0}}}
\renewcommand{\1}{\,{\texttt{\large 1}}}
\newcommand{\2}{\,{\texttt{\large 2}}}
\newcommand{\D}{\,{\texttt{\large D}}}
\renewcommand{\-}{\,{\texttt{\small $\bullet$}}}
\newcommand{\ppp}{\,{\texttt{\small $\bullet$}}}
\renewcommand{\u}{\underline}

\section{Computational tool in the case of a free action}\label{sec-tool-basic}

Assume that $(X,\mu)$ is a compact abelian group with the normalized Haar measure which is Pontryagin-dual to a countable discrete abelian group $A$. Assume furthermore that the action $\Ga\actson X$ is by continuous group automorphisms. 
The Pontryagin duality gives us  an embedding $\C[A]\into L^\infty(X)$. The preimage of $f\in L^\infty(X)$ under this embedding, if it exists, is denoted by $\wh f$ (see \cite[Chapter 4]{Folland:A_course_in_abstract_harmonic_analysis} for more on the Pontryagin duality).  

Let $\chi_1,\ldots, \chi_n$ be the indicator functions of subsets $X_1,\ldots, X_n\subset X$ such that all  $\chi_i$ have preimages. Let $a_1,\ldots, a_n\in \C$ and $\ga_1,\ldots, \ga_n \in \Ga$.

Let $\wh T\in \C[\Ga\ltimes A]$ be defined as $\wh T:= \sum a_i\ga_i\wh\chi_i$, and let $T\in\Ga\ltimes  L^\infty(X)$ be defined as $T := \sum a_i\ga_i\chi_i$. 

We consider $\C[\Ga\ltimes A]$ as acting on $l^2(\Ga\ltimes A)$ by bounded operators. The spectral measures of the elements of $\C[\Ga\ltimes A]$ are computed with respect to this action and the vector in $l^2(\Ga\ltimes A)$ which is the indicator function of the neutral element.

Similarly the \textit{group-measure space von Neumann algebra} $\Ga\ltimes L^\infty (X)$  (see e.g.~\cite[Chapters 1 and 2]{Lueck:Big_book}) acts on the direct integral Hilbert space $\int_X^\oplus l^2(\Ga)\,d\mu(x)$, and the spectral measure is computed with respect to the vector equal to the function which sends all $x\in X$  to the indicator function of the neutral element.

As explained for example in \cite[Section 2]{grabowski-on-turing-dynamical-systems-and-the-atiyah-problem}, we have the following lemma.

\begin{lemma}\label{lem-pontryagin}
The spectral measures of $\wh T$ and of $T$ are the same. \qed
\end{lemma}

We will now explain how to compute the spectral measure of $T$ \textit{under the assumption that the action $\Ga\actson X$ is essentially free}, i.e.~there is a subset $X'\subset X$ of full measure which is $\Ga$-invariant and such that the action of $\Ga$ on $X'$ is free.

Consider the oriented edge-labelled graph $\cal G$ defined as follows. The set of vertices of $\cal G$ is $X$, and there is an edge from $x_1$ to $x_2$ if for some $i$ we have $x_1\in X_i$ and $\ga_i.x_1=x_2$. On such an edge we set the label to be equal to 
\def\mathclap#1{\text{\hbox to 0pt{\hss$\mathsurround=0pt#1$\hss}}}
$$
	\sum_{\mathclap{\substack{j\colon \ga_j=\ga_i\\x_1\in X_j}}} a_j.
$$

 Let $\cal G(x)$ be the connected component of $x$ in $\cal G$. Let $l^2(\cal G(x))$ be the Hilbert space spanned by the vertices of $\cal G(x)$. Let $T(x) \colon l^2(\cal G(x)) \to l^2(\cal G(x))$ be the adjacency operator on $\cal G(x)$, i.e.~the entry of the matrix of $T(x)$ corresponding to a pair of vertices $(v_1, v_2)$ is equal to the label on the edge from $v_1$ to $v_2$, if there is such an edge, and $0$ otherwise.

We say $T$ is \textit{self-adjoint} if the set of those $x$ for which the matrix of $T(x)$ is Hermitian is of measure $1$. The next proposition follows from \cite[Proposition 2.10]{grabowski-on-turing-dynamical-systems-and-the-atiyah-problem}.

\begin{proposition}\label{prop-tool-free}
Let us assume that $T$ is self-adjoint and that the set of those $x$ for which $\cal G(x)$ is finite is of measure $1$. Then for a measurable subset $D\subset \R$ we have 
$$
\mu_T(D) = \int_X \frac{\mu_{T(x)}(D)}{|\cal G(x)|}\,d\mu(x).
$$\qed
\end{proposition}

We will apply this proposition in the next  section. Its utility comes from the fact that among the labelled graphs $\cal G(x)$, $x\in X$, there are  only countably many different ones, and they can be computed explicitly. As such the above integral will decompose as an explicit countable sum of spectral measures of \textit{finite-dimensional} matrices.

\section{Possible values of the Novikov-Shubin invariants}\label{sec-ns}

We need a more quantitative version of \cite[Lemma 5]{2014arXiv1409.3212G}. For $b\in \R$ and $n\in \N$ let $M(b,n)$ be the $n\times n$ matrix
$$
\left( \begin{array}{cccccc}
1 & b &  &  &   \\
b & b^2+1 & b &   &   \\
 & b & b^2+1  &  &   \\
& & & \ldots & & \\
 & & &  & b^2+1 &b \\ 
  & & &  &b & b^2+1
\end{array} \right)
$$

\begin{lemma}\label{lem-matrix-lemma}
For every $\eps$ and $b>1$ there is $N$ such that for $n>N$ the matrix $M(b,n)$ has an eigenvalue $\la_1(b,n)$ such that 
$$
(\frac{1}{b^2}-\eps)^n <\la_1(b,n) <(\frac{1}{b^2}+\eps)^n,
$$
and such that all the other eigenvalues are bigger than or equal to $(b-1)^2$. 
\end{lemma}
\begin{proof}
Let us fix $\eps$ and $b$. Let $K(b,n) = M(b,n) + \diag(b^2,0,0,\ldots,0)$, i.e.~we replace the anomalous $1$ on the diagonal with $b^2+1$. Let $\ka_1\le \ka_2\le\ldots\le \ka_n$ be the eigenvalues of $K(b,n)$ and let $\la_1\le \la_2\le \ldots \la_n$ be the eigenvalues of $M(b,n)$. Note that the norm of the matrix $K(b,n) - (b^2+1)\textrm I_m$ is $2b$, so we have the following claim.

\vspace{5pt}
\textbf{Claim A.} All the eigenvalues of $K(b,n)$ lie between $b^2+1-2b=(b-1)^2$ and $b^2+1+2b=(b+1)^2$.\qed
\vspace{5pt}

Let $D_n = \det(K(b,n))$ and $E_n = \det(M(b,n))$. By expanding both determinants along the final row we see that $D_n$ and $E_n$ fulfil the recurrence relations
$$
	D_{n+2} = (b^2+1)D_{n+1} - b^2 D_n\quad  	E_{n+2} = (b^2+1)E_{n+1} - b^2 E_n.
$$
Solving the recurrence in the standard way gives us $E_n=1$ for all $n$ and 
\beq\label{eq-det}
	D_n = \frac{b^2}{b^2-1} b^{2n} - \frac{1}{b^2-1}.
\eeq

Note that for any constant $C> 0$ we have that for sufficiently large $n$ the following holds:
\beq\label{eq-det-better}
(b^2-\eps)^n \le C D_n \le (b^2+\eps)^n
\eeq

Note that the difference $M(b,n)-K(b,n)$ is a rank $1$ matrix, so we can use the Weyl inequality for rank $1$ perturbations (e.g.~\cite[Theorem 4.3.4]{MR1084815}), which in particular implies that for $i= 2,\ldots,n$ we have $\la_i\ge \ka_{i-1}$. Since $\ka_1\cdot\ldots\cdot\ka_n =D_n$, it follows that $\la_2\cdot\ldots\cdot \la_n\ge \frac{D_n}{\ka_n}$.

Similarly the Weyl inequality implies that for $i=2,\ldots, n-1$ we have $\la_i\le \ka_{i+1}$, so that $\la_2\cdot\ldots\cdot\la_{n}\le \frac{\la_nD_n}{\ka_1\ka_2}$. 

Note that the norm of $ M(b,n) =K(b,n)- \diag(b^2,0,0,\ldots,0)$ is at most $(b+1)^2 +b^2$, so in particular $\la_n\le (b+1)^2+b^2$. This, together with Claim A, shows 
$$
\frac{D_n}{(b+1)^2} \le \la_2\cdot\ldots\cdot\la_n  \le \frac{((b+1)^2+b^2)D_n}{(b-1)^4}.
$$

Now \eqref{eq-det-better} implies that for sufficiently large $n$ we have
$$
(b^2-\eps)^n \le (\la_2\cdot\ldots\cdot \la_n) \le (b^2+\eps)^n.
$$

Finally since $\la_1\cdot\ldots\cdot \la_n = E_n=1$ we obtain
$$
    \frac{1}{(b^2+\eps)^n} \le \la_1 \le \frac{1}{(b^2-\eps)^n},
$$
which implies the statement about $\la_1$.

As for all the other eigenvalues, by Claim A we have $\ka_1\ge (b-1)^2$, and for $i\ge 2$ we have $\la_i\ge \la_2\ge \ka_1$ by the Weyl inequality, which finishes the proof.
\end{proof}

We introduce the following notation for the subsets of $\ZZ_2^\Z$. The elements of $\ZZ_2$ are denoted by $\0$ and $\1$. For 
 $\eps_i\in \{\0,\1\}$ we denote the set 
$$
	\{(m_i)\in \ZZ_2^\Z\colon m_{-a} = \eps_{-a}, \ldots, m_b = \eps_b\}\subset \ZZ_2^\Z,
$$
by
$$
  [\eps_{-a}\eps_{-a+1}\ldots\eps_{-1}\underline{\eps_{0}}\eps_1\ldots\eps_b],
$$
and we let
$$
  \chi[\eps_{-a}\eps_{-a+1}\ldots\eps_{-1}\underline{\eps_{0}}\eps_1\ldots\eps_b, x] \in L^\infty(\ZZ_2^\Z)
$$ 
be the corresponding indicator function. Elements from the set above will be denoted with the curly brackets $()$ instead of $[]$.

Recall that $t$ is the generator of the infinite cyclic group $\ZZ$. For $b\in \R$ let $T(b)\in \ZZ \ltimes L^\infty(\ZZ_2^\Z)$ be defined as 
$$
T(b) := -b^2\chi[\1\u\0] + b(t[\u\0]+t^{-1}\chi[\0\u\ast]) +(b^2+1).
$$
In this notation, the operator studied in \cite{Grigorchuk_Zuk2001} was $t[\u\0]+t^{-1}\chi[\0\u\ast]$. Note that the indicator functions in the definition of $T(b)$ are in the image of the Pontryagin duality map $\Q[\oplus_\Z \ZZ_2]\into L^\infty(\ZZ_2^\Z)$. So, by Lemma \ref{lem-pontryagin}, the Novikov-Shubin invariant of $T(b)$ is the same as the Novikov-Shubin invariant of the corresponding $\wh{T(b)}\in \R[\ZZ \ltimes \oplus_\Z \ZZ_2 ]= \R[\ZZ_2\wr \ZZ]$.

\begin{theorem}
For $b>1$ the Novikov-Shubin invariant of $T(b)$ is equal to  $\frac1{2\log_{2} (b)}$.
\end{theorem}
\begin{proof}
We use Proposition \ref{prop-tool-free} with $X=  \ZZ_2^\Z$, $\Ga=\ZZ$, $A=\oplus_\Z \ZZ_2$. Let us compute two examples of a graph $\cal G(x)$.

First let $x= (\1\u \1)$. Then $x\notin [\1\u\0]$, $x\notin  [\u\0]$ and $x\notin [\0\u\ast]$, so the only outgoing arrow at $x$ is the self-loop  with label $(b^2+1)$. 

As for the incoming arrows at $x$, other than the self-loop, we see that $x\notin t.[\u\0]$ and $x\notin t^{-1}.[\0\u\ast]$, so there are no incoming arrows. Accordingly $\cal G(x)$ consists only of the vertex $x$ with a self-loop with label $b^2+1$.

Now let $x= (\1\u\0\0\1)$. Since $x\in  [\u\0]$ there is an outgoing arrow from $x$ to $t.x = (\1\0\u \0\1)$ with label $b$. Since $t.x\in [\u\0]$, there is an outgoing arrow from $t.x$ to $t^2.x = (\1\0\0\u\1)$ with label $b$. Since $t.x\in [\0\u\ast]$, there is also an outgoing arrow from $t.x$ to $x$ with label $b$. Similarly $t^2.x \in [\0\u\ast]$ so there is an arrow from $t^2.x$ to $t.x$ with label $b$.

As for the self-loops , $x\in [\1\u\0]$, so there is a self-loop at $x$ with label $(b^2+1)-b^2 = 1$. The vertices $t.x$ and $t^2.x$ have self-loops with labels $b^2+1$.

In analogy with these two examples we see that when $x\in [\1\u\0\0^k\1]$ then $\cal G(x)$ is the graph on Figure \ref{fig-app} with $k+2$ vertices.

\begin{figure}[here]
\centering
\begin{tikzpicture}
  [scale=1.1,auto=center]
  \tikzset{edge/.style = {->,> = latex'}}
  \node (n1) at (1,1) {$\bullet$};
  \node (n2) at (3,1)  {$\bullet$};
  \node (n3) at (5,1)  {$\ldots$};
  \node (n4) at (7,1) {$\bullet$};

  \foreach \from/\to in {n1/n2,n2/n3,n3/n4}
    \draw[edge, bend right=23] (\from) to node[below] {$b$} (\to);
  \foreach \from/\to in {n1/n2,n2/n3,n3/n4}
    \draw[edge, bend right=23] (\to) to node[above] {$b$} (\from);
  \foreach \from/\to in {n2/n2,n4/n4}
    \draw[loop above,thick,edge] (\from) to node {$b^2+1$} (\to);
    \draw[loop above,thick,edge] (n1) to node {$1$} (n2);
\end{tikzpicture}
\caption{}\label{fig-app}
\end{figure}

 Let us check that, up to a set of measure $0$, every point of $X$ is in a connected component of $
 \cal G(x)$ for some $x\in [\1\u\0\0^k\1]$: 
$$
\mu ([\1\u \1]) + \sum_{k=0}^\infty (k+2)\mu([\1\u\0\0^k\1]) = \frac{1}{4} + \sum_{k=0}^\infty (k+2) \frac{1}{2^{k+3}} = \frac{1}{2}\sum_{k=1}^\infty \frac{k}{2^k} = 1.
$$
In particular the subset of those $x$ for which $\cal G(x)$ is finite is of full measure. Clearly the adjacency operator on the graph with $m$ vertices on Figure \ref{fig-app} is given by the matrix $M(b,m)$. Proposition \ref{prop-tool-free} now shows that 
$$
	\mu_T = \frac{1}{4} \mu_{\diag(b^2+1)} + \sum_{m=2}^\infty \frac{1}{2^{m+1}} \mu_{M(b,m)}.
$$

Let us use Lemma \ref{lem-matrix-lemma} to estimate $\mu_T((0,z])$ for small $z>0$. Let us fix a small $\eps$ in Lemma \ref{lem-matrix-lemma}. Then for sufficiently small $z$ we have
\beq\label{eq-sum}
\mu_T((0,z]) = \sum_{m\colon \la_1(b,m)\le z} \frac{1}{2^{m+1}}.
\eeq
By Lemma \ref{lem-matrix-lemma}, the smallest $m$ such that $\la_1(b,m)\le z$  is between
$$
\frac{|\log(z)|}{|\log(\frac{1}{b^2}+\eps)|}
$$
and 
$$
\frac{|\log(z)|}{|\log(\frac{1}{b^2}-\eps)|}.
$$

We estimate $\mu_T((0,z])$ from (i) below and (ii) above by taking in the sum \eqref{eq-sum} respectively (i) only the smallest $m$ such that $\la_1(b,m)\le z$, and (ii) the smallest such $m$ and all the natural numbers larger than $m$. We obtain that  $\mu_T((0,z])$ lies between
$$
2^{\frac{\log(z)}{|\log(\frac{1}{b^2}-\eps)|}} = z^\frac1{|\log(\frac{1}{b^2}-\eps)|}
$$
and
$$
2\cdot 2^{\frac{\log(z)}{|\log(\frac{1}{b^2}+\eps)|}} =2z^\frac1{|\log(\frac{1}{b^2}+\eps)|} 
$$
(in the algebraic manipulations we used that $\log(z)$ is negative for small $z$).

This shows that the Novikov-Shubin invariant of $T(b)$ lies between $\frac{1}{|\log(\frac{1}{b^2}-\eps)|}$ and  $\frac{1}{|\log(\frac{1}{b^2}+\eps)|}$, for every $\eps$, and so in fact must be equal to $\frac1{|\log(\frac{1}{b^2})|} = \frac{1}{2\log(b)}$. 
\end{proof}

\section{Computational tool in the case of a non-free action}\label{sec-tool-non-basic}

We will now repeat the discussion from Section \ref{sec-tool-basic}, and add some extra structure in order to deal with a non-free action. For the proofs see  \cite[Section 2]{grabowski-on-turing-dynamical-systems-and-the-atiyah-problem}.

Let $\Ga\actson X$ be as in Section \ref{sec-tool-basic}, with the exception that it is not necessarily a free action. Let $T\in \Ga\ltimes L^\infty(X)$ be defined as $T := \sum_{i=1}^n a_i\ga_i\chi_i$ (with the notation from Section \ref{sec-tool-basic}).

Consider the oriented graph $\cal G_\Ga$ whose set of vertices is $X$, and with edges labelled by the elements of the set  $\{\ga_1,\ldots,\ga_n\}$, defined as follows. There is an edge with label $\ga_i$ from $x_1$ to $x_2$ if $x_1\in X_i$ and $\ga_i.x_1=x_2$. Let $\cal G_\Ga(x)$ be the connected component of $x$.  We say $\cal G_\Ga(x)$ is \textit{simply-connected} if multiplying edge-labels along any closed loop gives the trivial element of $\Ga$ (if a loop traverses an edge in the direction opposite to the orientation of the edge, we invert the label). 

Let $\cal G(x)$ be the graph which arises from $\cal G_\Ga(x)$ by changing the  label $\ga_i$ on the edge between $x_1$ and $x_2$ as above to the sum
$$
	\sum_{\mathclap{\substack{j\colon \ga_j=\ga_i\\x_1\in X_j}}} a_j.
$$
Finally let  $T(x) \colon l^2(\cal G(x)) \to l^2(\cal G(x))$ be the adjacency operator on $\cal G(x)$. The next proposition follows from \cite[Proposition 2.10]{grabowski-on-turing-dynamical-systems-and-the-atiyah-problem}.

\begin{proposition}\label{prop-tool-non-free}
Let us assume that the set of $x$ such that  $\cal G_\Ga(x)$ is finite and simply-connected  is of full measure. Then $\dimvn\ker T$ is equal to 
$$
\int_X \frac{\dim\ker{T(x)}}{|\cal G(x)|}\,d\mu(x).
$$\qed
\end{proposition}

\section{Irrational $l^2$-Betti numbers arising from $\ZZ_p\wr \ZZ$}\label{sec-the-operator}

For the rest of the article let $X$ be the compact abelian group $\ZZ_p^\Z\times \ZZ_2^3$, and $\Ga = \ZZ \times \aut(\ZZ_2^3)$. The action $\Ga {\actson} X$ is the natural one, i.e.~$\aut(\ZZ_2^3)$ acts on $\ZZ_2^3$ and $\ZZ$ acts on $\ZZ_p^\Z$ by shifting the coordinates.

Note that  $\Ga\ltimes A $ is isomorphic to $(\ZZ_p \wr \ZZ) \times (\aut(\ZZ_2^3)\ltimes \ZZ_2^3)$. We will shortly define $T \in \Q[\Ga\ltimes A]$ such that 
$$
\dimvn\ker T = \frac{4p^3+3p^2+2p-1}{8p^3} + \frac1{8p^2(p-1)} \sum_{k=1}^\infty (\frac{p-1}{p})^{k+2^{k-1}},
$$

The additional factor $1344$ in Theorem \ref{thm-intro-zero} comes from the fact that $\ZZ_p\wr\ZZ$ is a  subgroup in $\Ga\ltimes A$ of index $1344$ (see e.g.~\cite[Lemma 6.2]{grabowski-on-turing-dynamical-systems-and-the-atiyah-problem} for more explanation). Furthermore, for $k\neq 0$ the kernels of $T$ and $kT$ are the same, so we will also obtain a matrix over $\Z[\Ga\ltimes A]$ whose kernel dimension is as above.

Let $A$, $B$, $C$, $D$, $F$, $I$, $U_1$, $U_2$ ($U$ stands for \textit{unimportant}, $F$ for \textit{final} and $I$ for \textit{initial}) denote the elements of $\Zmod{2}^3$. The only assumption on this labelling is  that the first $6$ symbols correspond to non-zero elements of $\Zmod{2}^3$.

For every pair $(x,y)$ of different elements from the set $\{A, B, C,D, F, I\}$ we fix an automorphism denoted  by $(x\to y)\in \aut(\Zmod{2}^3)$ which sends
$x$ to $y$, in such a way that  
\begin{equation}\label{eq_loops_inv}
(x\to y) = (y\to x)^{-1}
\end{equation}
and
\begin{equation}\label{eq_loops_init}
(C\to D)(A\to C)=(I\to D)(A\to I).
\end{equation}
To treat the case of an arbitrary $p$, we change our notation in the following way. 
Let $\0 := \{0\} \subset \ZZ_p$ and $\1 := \{1,2,3,\ldots, p-1\}\subset \ZZ_p$. Let
$$
  [\eps_{-a}\eps_{-a+1}\ldots\eps_{-1}\underline{\eps_{0}}\eps_1\ldots\eps_b, x],
$$
where $\eps_i\in \{\0,\1\}$, denote  
$$
	\{((m_i), y)\in \Zmod{p}^\Z \times \Zmod{2}^3: m_{-a} \in \eps_{-a}, \ldots, m_b \in \eps_b, y= x\}\subset X,
$$
and let  
$$
  \chi[\eps_{-a}\eps_{-a+1}\ldots\eps_{-1}\underline{\eps_{0}}\eps_1\ldots\eps_b, x] \in L^\infty(X)
$$ 
be the corresponding indicator function.

Let $S \in \Q[\Ga\ltimes A]$ be represented by the sum of the following terms:
\begin{align}
 \label{eq_Sdef_1} ( -t(I\to D) + t^{-1}(I\to A)) &\cdot \chi[\1\underline{\0}\1,I]		\\
  (-t^2(A\to C) - 2t^{-1} ) &\cdot\chi[\1\underline{\1}\0\1,A]	\nonumber		\\
 -t^2(A\to C)		&\cdot	\chi[\0\underline{\1}\0\1,A]\nonumber	\\
 -2t^{-1} &\cdot\chi [\1\underline{\1}\0\0,A]			\nonumber	\\
 0 &\cdot \chi[\0\underline{\1}\0\0,A]				\nonumber\\
  -2t^{-1} &\cdot \chi[\1\underline{\1}\1,A]				\nonumber\\
 -(A\to B) &\cdot \chi[\0\underline{\1}\1,A]			\nonumber	\\
  -t  &\cdot \chi[\underline{\1}\1,B]		\nonumber		\\
  -(B\to A) &\cdot \chi[\underline{\1}\0,B]		\nonumber		\\
  (-t + (C\to D)) &\cdot \chi[\underline{\1}\1,C]	\nonumber			\\
 +(C\to D)  &\cdot \chi[\underline{\1}\0,C]		\nonumber		\\
 -t   &\cdot \chi[\underline{\1}\1,D]		\nonumber		\\
  -(D\to F) &\cdot \chi[\underline{\1}\0,D]		\nonumber		\\
 0	&\cdot \chi[\underline{\1}\0,F]		\nonumber		\\
 0 		&\cdot \chi_R,					\nonumber
\end{align}
where $\chi_R$ is the indicator function of the set $R$ defined to be  ``all the rest'', i.e.~the complement of the union of the sets $[\1\underline{\0}\1,I]$,
$[\1\underline{\1}\0\1,A]$, $[\0\underline{\1}\0\1,A]$, $[\1\underline{\1}\0\0,A]$, $[\0\underline{\1}\0\0,A]$,
$[\1\underline{\1}\1,A]$, $[\0\underline{\1}\1,A]$, $[\underline{\1}\1,B]$, $[\underline{\1}\0,B]$, $[\underline{\1}\1,C]$.
$[\underline{\1}\0,C]$, $[\underline{\1}\1,D]$,  $[\underline{\1}\0,D]$ and $[\underline{\1}\0,F]$. 

Finally define
\begin{equation}\label{eq_Tdef}
	T:=S + 1-\chi_R - \chi[\1\underline{\0}\1,I] - \chi[\underline{\1}\0,F]
\end{equation}

\begin{remark}\label{rem-it-is-easy}
(i) The reason we explicitly write the  terms "$0\scdot \ldots$" is that this way the right hand sides are indicator functions of disjoint sets whose union is $X$. This is helpful when checking that the connected components $\cal G_\Ga(x)$ are as claimed. To reassure the reader, without any $0$-terms it would be the same operator and the same computations would have to be performed.

(ii) The definitions of $S$ and $T$ might seem complicated at first. Let us informally describe how the author came up with them. In the process of finding a group ring element over $\ZZ_2\wr \ZZ$ (or a matrix of group ring elements) whose kernel dimension is irrational, the first step was a realization that any family of \textit{simple-to-describe} graphs can appear as the connected components $\cal G(x)$. Examples of simple-to-describe graphs are on Figures \ref{fig_sgraph_g}, \ref{fig_sgraph_h} and \ref{fig_sgraph_j}; one could formalize the notion of being simple-to-describe using regular languages. Then it was necessary to find a simple-to-describe family whose kernel dimensions behave in an irregular way. This was the most difficult step - after trial and error the family from Figure \ref{fig_sgraph_j} was found. The operator $T$ above is defined in such a way so that that family appears among the connected components $\cal G (x)$ (two other infinite families, those from Figures \ref{fig_sgraph_g} and \ref{fig_sgraph_h} also appear, but their kernel dimensions behave in a regular way, so they do not interefere with the irregularity of the family from Figure \ref{fig_sgraph_j}).

\end{remark}

We will now describe the  graphs $\cal G_\Ga(x)$  and $\cal G(x)$ for  $x\in X$. It is convenient to describe them in four families, which we do in separate subsections. 

We will show figures for the graphs, but for clarity we suppress self-loops. Note that the self-loops are given only by the terms in  \eqref{eq_Tdef}, so it is also easy to take them into account.

In all the cases it is somewhat tedious but, using Remark \ref{rem-it-is-easy},  straightforward to check that the graph $\cal G_\Ga(x)$ is as claimed for a given $x\in X$.

\subsection{Case 1: $x\in R$}\label{subsec_sgraph_u}\mbox{}

The graph $\cal G_\Ga(x)$ consists of just one vertex with no edges. Accordingly, the adjacency operator $T(x)$ is the $0$ operator. We clearly deduce the following lemma.


\begin{lemma}We have the following properties.\label{lemma_sgraph_u}
\mbox{}
\begin{enumerate}
 \item $\dim\ker T(x)  = 1$.
 \item $\cal G_\Ga(x)$ is simply-connected. 
 \item $\mu (R) =\frac{1}{8}(2+5\frac1p+ \frac{1}{p^{3}} + 2\frac{p-1}{p^3} + \frac{p-1}{p} + (\frac{p-1}{p})^2) $
\end{enumerate}
\end{lemma}
\begin{proof}
(1) and (2) are clear. As for (3), note that we can explicitly write
\begin{align*}
R &= [\underline{\0}, A] \sqcup [\underline{\0}, B] \sqcup [\underline{\0},C] \sqcup [\underline{\0}, D] \sqcup [\cdot, U_1] \sqcup [\cdot, U_2] \sqcup\\ & \sqcup [\underline{\0}, F] \sqcup [\underline{\1}\1, F] \sqcup [\underline{\1}, I] \sqcup
[\1\underline{\0}\0, I] +\sqcup [\0\underline{\0}\1, I] \sqcup [\0\underline{\0}\0,I].
\end{align*}
Since $\mu$ is the product measure, it is easy to compute the measures of the sets above. We start with  $\mu([\u\0]) = \frac{1}{p}$, $\mu([\u\1]) = \frac{p-1}{p}$, and then for example $\mu([\0\uz\1, I]) = (\frac{1}{p})^2\scdot \frac{p-1}{p}\scdot \frac{1}{8}$.
\end{proof}

\subsection{Case 2: $x\in [\0\uo\1^{k-1}\0\0,A]$}\label{subsec_sgraph_g}\mbox{} 

If we denote $x =(\0\uo\1^{k-1}\0\0,A)$, then the vertices of $\cal G_\Ga(x)$ are  
\begin{align*}
(\0\uo\1^{k-1}\0\0,A)\quad (\0\1\uo\1^{k-2}\0\0,A)\quad \ldots,\quad (\0\1^{k-1}\uo\0\0,A) \\
(\0\uo\1^{k-1}\0\0,B) \quad (\0\1\uo\1^{k-2}\0\0,B)\quad \ldots\quad (\0\1^{k-1}\uo\0\0,B).
\end{align*}

$\cal G_\Ga(x)$ is shown on Figure \ref{fig_sgraph_g_emb}. Each vertex should additionally have a self-loop with label $e$. To avoid clutter only some vertices are explicitly identified as elements of $X$.

To facilitate to the reader checking that $\cal G_\Ga(x)$ is as claimed we indicate that the corresponding terms in \eqref{eq_Sdef_1}  are
\begin{align*}
[\0\uo\1,A]& \quad[\1\uo\1,A]&\quad\ldots&\quad [\1\uo\1,A]&\quad[\1\uo\0\0, A]& \\
[\uo\1,B]& \quad[\uo\1,B]&\quad\ldots&\quad [\uo\1,B]&\quad[\uo\0, B]&.
\end{align*}

The graphs $\cal G (x)$ are shown  on Figure \ref{fig_sgraph_g}. Each vertex should additionally have a self-loop with label $1$. 

\begin{figure}[h]%
  \resizebox{0.8\textwidth}{!}{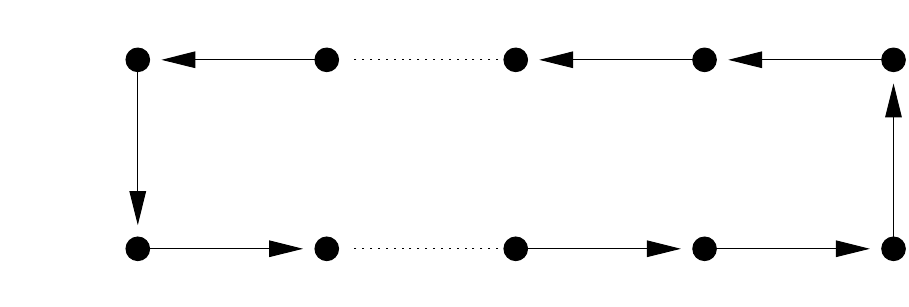}
  \caption{$\cal G_\Ga(x)$ without self-loops for $x= (\0\uo\1^{k-1}\0\0,A)$.}
  \label{fig_sgraph_g_emb}
\end{figure}

\begin{figure}[h]%
  \resizebox{0.84\textwidth}{!}{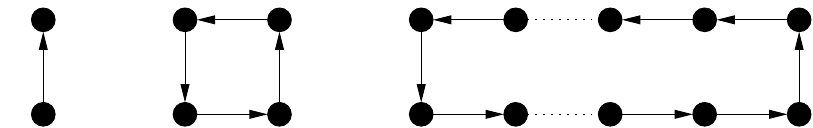}
  \caption{$\cal G(x)$ without self-loops for $x= (\0\uo\0\0,A)$, $x= (\0\uo\1\0\0,A)$, and $x= (\0\uo\1^{k-1}\0\0,A)$.}
  \label{fig_sgraph_g}
\end{figure}

\begin{lemma}\label{lemma_sgraph_g}We have the following properties.
\begin{enumerate}
 \item $\dim \ker T(x) = 0$
 \item $\cal G_\Ga(x)$ is simply-connected. 
 \item \textsc{$\mu ( [\0\uo\1^{k-1}\0\0,A]) = \frac18\cdot (\frac{1}{p})^3\cdot (\frac{p-1}{p})^k$} and $|\cal G(x)| = 2k$.
\end{enumerate}
\end{lemma}
\begin{proof}
(2) follows easily from Figure \ref{fig_sgraph_g_emb} and Equation \eqref{eq_loops_inv}. (3) is a direct computation as in Lemma \ref{lemma_sgraph_u}. (1) follows from analysing Figure \ref{fig_sgraph_g}, but for completeness we give a proof in the appendix.
\end{proof}

\subsection{Case 3: $x\in [\0\0\uo\1^{l-1}\0,C]$}\label{subsec_sgraph_h}\mbox{}

If we denote $x =(\0\0\uo\1^{l-1}\0,C)$ then the vertices of $\cal G_\Ga(x)$ are 	
\begin{align*}
(\0\0\uo\1^{l-1}\0,C)&\quad  \ldots&\quad (\0\0 \1^{l-1}\uo\0,C)& \\
(\0\0\uo\1^{l-1}\0,D)& \quad \ldots& \quad (\0\0 \1^{l-1}\uo\0,D)& \\
&&(\0\0 \1^{l-1}\uo\0,F)&
\end{align*}

$\cal G_\Ga(x)$ is shown on Figure \ref{fig_sgraph_h_emb}. Each vertex except the final one should additionally have a self-loop with label $e$. To avoid clutter only some vertices are explicitly identified as elements of $X$.

To facilitate to the reader checking that $\cal G_\Ga(x)$ is as claimed we indicate that the corresponding terms in \eqref{eq_Sdef_1}  are
\begin{align*}
[\uo\1,C],\quad \ldots,\quad [\uo\1,C],\quad [\uo \0,C], \\
[\uo\1,D], \quad\ldots, \quad [\uo\1,D], \quad[\uo \0,D], \\
[\uo \0, F].
\end{align*}

The graphs $\cal G (x)$ are shown  on Figure \ref{fig_sgraph_h}.  Each vertex except the final one should additionally have a self-loop with label $1$. 

\begin{figure}[h]%
  \resizebox{0.62\textwidth}{!}{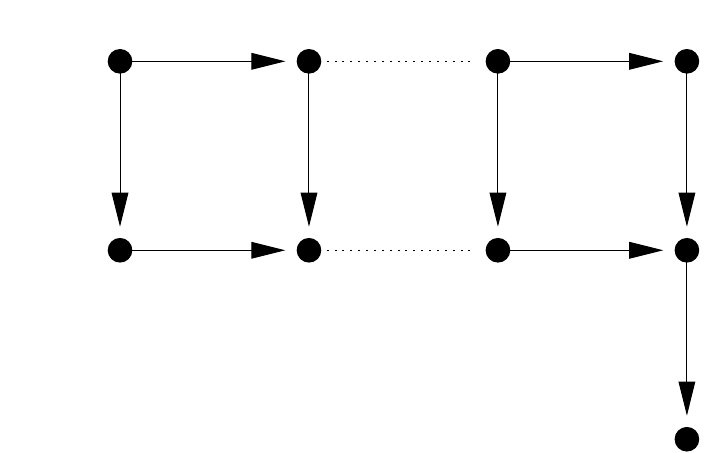}
  \caption{$\cal G_\Ga(x)$ without self-loops for $x= (\0\0\,\uo\1^{l-1}\0,C)$.}
  \label{fig_sgraph_h_emb}
\end{figure}

\begin{figure}[h]%
  \resizebox{0.74\textwidth}{!}{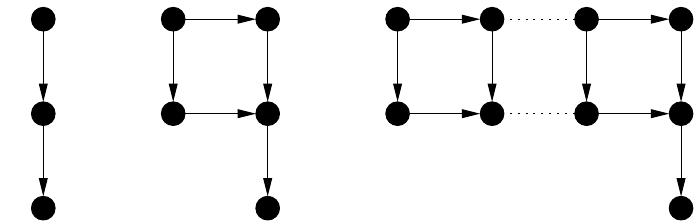}
  \caption{$\cal G(x)$ without self-loops for $x= (\0\0\,\uo\0,C)$, $x= (\0\0\,\uo\1\0,C)$, and $x= (\0\0\,\uo\1^{l-1}\0,C)$}
  \label{fig_sgraph_h}
\end{figure}

\begin{lemma} The following properties are true.\label{lemma_sgraph_h}\begin{enumerate}
 \item $\dim \ker T(x) = 1$
 \item $\cal G_\Ga(x)$ is simply-connected. 
 \item \textsc{$\mu ( [\0\0\uo\1^{l-1}\0,C]) = \frac18\cdot (\frac{1}{p})^3 \cdot (\frac{p-1}{p})^l$} and $|\cal G(x)| = 2l+1$.
\end{enumerate}
\end{lemma}
\begin{proof}
(2) follows easily from Figure \ref{fig_sgraph_h_emb} and Equation \eqref{eq_loops_inv}. (3) is a direct computation as in Lemma \ref{lemma_sgraph_u}. (1) follows from analysing Figure \ref{fig_sgraph_h}, but for completeness we give a proof in the appendix.
\end{proof}

\subsection{Case 4: $x\in [\0\uo\1^{k-1}\0\1^l\0,A]$}\label{subsec_sgraph_j} \mbox{}

If we denote $x=(\0\uo\1^{k-1}\0\1^l\0,A)$ then the vertices of $\cal G_\Ga(x)$ are 
\begin{align*}
(\0\uo\1^{k-1}\0\1^l\0,A),\,\,  (\0\1\uo\1^{k-2}\0\1^l\0,A),\,\,  \ldots,\,\, (\0\1^{k-1}\uo\0\1^l\0,A), \\
(\0\uo\1^{k-1}\0\1^l\0,B),\,\,  (\0\1\uo\1^{k-2}\0\1^l\0,B), \,\, \ldots,\,\,  (\0\1^{k-1}\uo\0\1^l\0,B),\\
(\0\1^{k}\uz\1^l\0,I),\\
(\0\1^k\0\uo\1^{l-1}\0,C), \,\, \ldots,\,\,  (\0\1^k\0 \1^{l-1}\uo\0,C), \\
(\0\1^k\0\uo\1^{l-1}\0,D), \,\, \ldots,\,\,  (\0\1^k\0 \1^{l-1}\uo\0,D), \\
(\0\1^k\0 \1^{l-1}\uo\0,F).
\end{align*}

$\cal G_\Ga(x)$ is shown  on Figure  \ref{fig_sgraph_j_emb}. Each vertex except the final and the initial ones should additionally have a self-loop with label $e$. To avoid clutter only some vertices are explicitly identified as elements of $X$. Because it could be unclear which labels correspond to which vertices, the identified vertices are marked white. 

To facilitate to the reader checking that $\cal G_\Ga(x)$ is as claimed we indicate that the corresponding terms in \eqref{eq_Sdef_1}  are
\begin{align*}
[\0\uo\1,A],\quad[\1\uo\1,A],\quad\ldots,\quad[\1\uo\1,A],\quad[\1\uo\1\0\1, A], \\
[\uo\1,B],\quad[\uo\1,B],\quad\ldots,\quad[\uo\1,B],\quad[\uo\0, B],\\
[\1\underline{\0}\1, I],\\
[\uo\1,C],\quad \ldots, \quad[\uo\1,C], \quad[\uo \0,C], \\
[\uo\1,D], \quad\ldots, \quad[\uo\1,D],\quad [\uo \0,D], \\
[\uo \0, F].
\end{align*}

The graphs $\cal G (x)$ are shown  on Figure \ref{fig_sgraph_j}.  Each vertex except the final and the initial ones should additionally have a self-loop with label $1$.
 
\begin{figure}[h]%
  \resizebox{0.87\textwidth}{!}{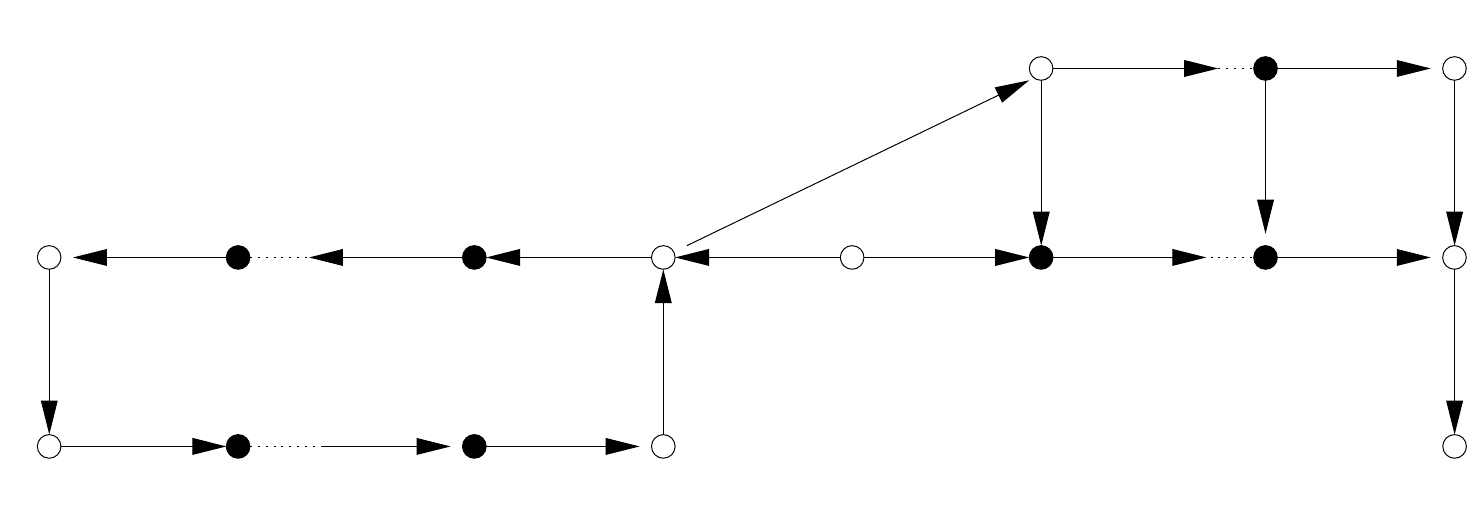}
  \caption{$\cal G_\Ga(x)$ without self-loops for $x= (\0\uo\1^{k-1}\0\1^l\0,A)$.}  \label{fig_sgraph_j_emb}
\end{figure}

\begin{figure}[h]%
  \resizebox{0.95\textwidth}{!}{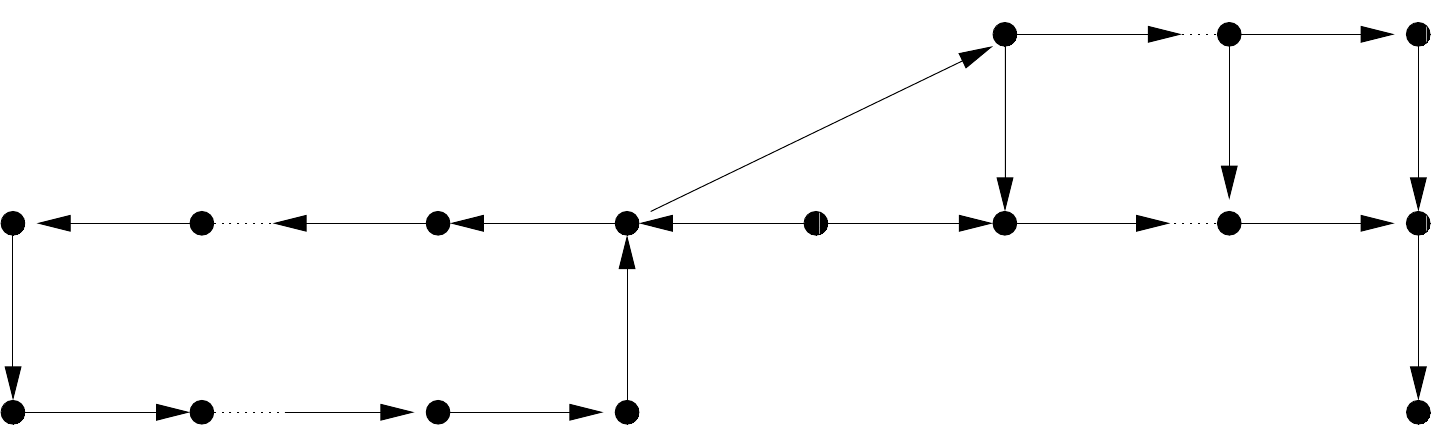}
  \caption{$\cal G(x)$ without self-loops for $x= (\0\uo\1^{k-1}\0\1^l\0,A)$}
  \label{fig_sgraph_j}
\end{figure}

\begin{lemma}The following properties are true.\label{lemma_sgraph_j}
\begin{enumerate}
 \item $\dim \ker T(x) = \left\{ 
	\begin{array}{l l}
  		2 & \quad \mbox{if $l=2^{k-1}-1$ }\\
		1 & \quad \mbox{otherwise}\\ 
	\end{array} \right. $
 \item $\cal G_\Ga(x)$ is simply-connected. 
 \item \textsc{$\mu ([\0\uo\1^{k-1}\0\1^l\0,A]) = \frac18\cdot (\frac{1}{p})^3\cdot (\frac{p-1}{p})^{k+l}$} and $|\cal G(x)| = 2k+2l+2$.
\end{enumerate}
\end{lemma}
\begin{proof}
(2) follows easily from Figure \ref{fig_sgraph_j_emb} and Equations \eqref{eq_loops_inv} and \eqref{eq_loops_init}. (3) is a direct computation as in Lemma \ref{lemma_sgraph_u}. (1) follows from analysing Figure \ref{fig_sgraph_h}, but for completeness we give a proof in the appendix.
\end{proof}

\subsection{Checking that we have not missed any graphs}\mbox{}

We need to check that the graphs $\cal G(x)$ on Figures \ref{fig_sgraph_g_emb}, \ref{fig_sgraph_h_emb} and \ref{fig_sgraph_j_emb}, together with the set $R$ cover the whole space $X$. To this end we compute that the measure of the covered part is $1$, by using the formulas in Lemmas \ref{lemma_sgraph_u}(3), \ref{lemma_sgraph_g}(3), \ref{lemma_sgraph_h}(3) and \ref{lemma_sgraph_j}(3).

Let $\al :=\frac{1}{p}$, $\be := \frac{p-1}{p}$. We need to check that  
\begin{multline*}
\frac{1}{8}(2+5\al+ \al^3 + 2\be\al^2 + \be + \be^2) \,\,+\,\,
\sum_{k=1}^\infty 2k\scdot \frac18\scdot \al^3\scdot \be^k \,\,+\,\,\\
+\sum_{l=1}^\infty (2l+1)\scdot \frac18\scdot \al^3 \scdot \be^l \,\,+\,\,
\sum_{k,l=1}^\infty (2k+2l+2) \scdot \frac18\scdot \al^3\scdot \be^{k+l} \,=\,1.
\end{multline*}

This is a tedious but elementary exercise in using the formula 
$$
	\sum_{n=1}^\infty (n+C) x^n = \frac{x}{(1-x)^2} + \frac{Cx}{1-x},
$$
valid for $0\le x \le 1$.

\subsection{The end game}\label{subsec-the-end-game}\mbox{}

We are now in a position to use Proposition \ref{prop-tool-non-free}. The following corollary, together with the discussion at the beginning of Section \ref{sec-the-operator}, proves Theorem \ref{thm-intro-zero}.

\begin{cory}\label{cory_dimT} We have
$$
\dimvn\ker T =
\frac{4p^3+3p^2+2p-1}{8p^3} + \frac1{8p^3} \sum_{k=1}^\infty (\frac{p-1}{p})^{k+2^{k}},
$$ which is a transcendental number.
\end{cory}
\begin{proof}

Let $T_0$ be the $0$ operator $\C\to \C$, let $T_1(k)\colon \C^{2k}\to\C^{2k}$ be the adjacency operator on the graph from Figure \ref{fig_sgraph_g}, let $T_2(l)\colon \C^{2l+1}\to \C^{2l+1}$ be the adjacency operator on the graph from Figure \ref{fig_sgraph_h}, and finally let $T_3(k,l)\colon \C^{2k+2l+2} \to \C^{2k+2l+2}$ be the adjacency operator on the graph from Figure \ref{fig_sgraph_j}.

By Proposition \ref{prop-tool-non-free} and the computations in the previous subsections, the left-hand side is equal to the sum of the following terms
\begin{gather*}
 	\frac{1}{8}(2+5\al+ \al^3 + 2\be\al^2 + \be + \be^2) \cdot \dim\ker T_0, \\
 	\sum_{k=1}^\infty \frac18\cdot \al^3\cdot \be^k\cdot \dim\ker T_1(k), \\
 	\sum_{l=1}^\infty \frac18\cdot \al^3 \be^l \dim\ker T_2(l),  \\
 	\sum_{k,l=1}^\infty \frac18\cdot \al^3 \be^{k+l} \dim\ker T_3(k,l).
\end{gather*}
Substituting the values for the kernel dimensions we get
\begin{multline*}
	\frac{1}{8}(2+5\al+ \al^3 + 2\be\al^2 + \be + \be^2) \,+\,0 \,+\, \sum_{l=1}^\infty \frac18\cdot \al^3\be^l \,+\,\\
	 \,+\, \sum_{k,l=1}^\infty \frac18\cdot \al^3 \be^{k+l}  \,+\, \sum_{k=2}^\infty \frac18\cdot \al^3
\be^{k+2^{k-1}-1}.
\end{multline*}
Noting that  $\sum_{k,l=1}^\infty \be^{k+l} = \sum_k \be^k\sum_l\be^l = (\frac{\be}{\al})^2$, after a short calculation  we obtain
$$
	\frac{1}{8}(2+5\al+ \al^3 + 2\be\al^2 + \be + \be^2) + 
        \frac18\al^2\be + \frac18 \al\be^2 +
	\frac18 \al^3\sum_{k=1}^\infty \be^{k+2^{k}},
$$
which is equal to the right-hand side.

Transcendence of  $\sum_{k=1}^\infty (\frac{p-1}{p})^{k+2^{k-1}}$ follows from \cite[Theorem 1]{Tanaka:Transcendence_of_the_values_of_certain_series_with_Hadamard's_gaps}.  Although similar series have
been studied already by Mahler \cite{Mahler:Arithmetische_Eigenschaften_der_Losungen_einer_Klasse_von_Funktionalgleichungen}, the article  \cite{Tanaka:Transcendence_of_the_values_of_certain_series_with_Hadamard's_gaps} seems to be the first
work which implies the transcendence of $\sum_{k=1}^\infty (\frac{p-1}{p})^{k+2^{k-1}}$.
\end{proof}

\appendix
\section{Linear algebra computations}

The following obvious lemma will be used many times. 
\begin{lemma}[``flow lemma at a vertex $v$'']\label{lemma_flow}
Let $T$ be the adjacency operator on an edge-labelled directed graph, let $v$ be a vertex, let $w_1,\ldots, w_n$ be all the vertices for which there are directed edges towards $v$, and let the corresponding edge labels be $a_1,\ldots, a_n\ \in \C$. Let $f\in \ker T$. Then 
$$
	\sum a_i f(w_i) = 0.
$$	
\end{lemma}\qed

\subsection{$x\in [\0\uo\1^{k-1}\0\0,A]$}\label{sec_graph_g}\mbox{} 

We give the vertices of $\cal G(x)$ shorthand names as in Figure \ref{fig_graph_g}. 
\begin{figure}[h]%
  \resizebox{0.72	\textwidth}{!}{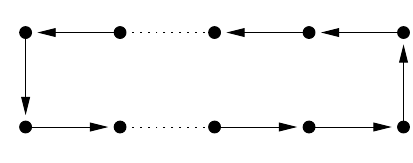}
  \caption{}  
  \label{fig_graph_g}
\end{figure}

\begin{lemma}\label{lem-appendix-g} We have
  $\dim\ker T(x) = 0$.
\end{lemma}
\begin{proof}
A direct check confirms the claim when $k=1$.  For $k>1$ let $f\in \ker T(x)$. From the flow lemma at $A_1$ we see that
$f(A_1) = f(B_k)$, and inductively $f(A_1) = f(B_1) = f(A_k)$.

On the other hand from the flow lemma at $A_2$ we see $f(A_2) = 2\scdot f(A_1)$, and inductively  $f(A_k) = 2^{k-1}\scdot f(A_1)$. Altogether we get
$$
f(A_1) = 2^{k-1}\scdot f(A_1),
$$ 
which is a contradiction. 

\end{proof}

\subsection{$x\in [\0\0\uo\1^{l-1}\0,C]$}\label{sec_graph_h}\mbox{} 

We give the vertices of $\cal G(x)$ shorthand names as in Figure \ref{fig_graph_h}. 
\begin{figure}[h]%
  \resizebox{0.56\textwidth}{!}{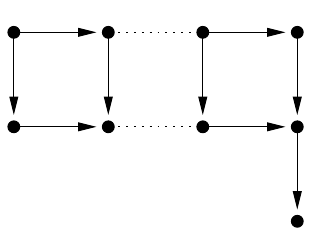}
    \caption{}
  \label{fig_graph_h}
\end{figure}

\begin{lemma}
 We have $\dim \ker T(x) = 1$.
\end{lemma}
\begin{proof}

The matrix of $T(x)$ in the basis $C_1,\ldots, C_l, D_1,\ldots, D_l, F$ is upper-triangular. The diagonal entries corresponding to $C_i$ and $D_i$ are equal to $1$, and the diagonal entry corresponding to $F$ is $0$. This shows the lemma.
\end{proof}

\subsection{$x\in [\0\uo\1^{k-1}\0\1^l\0,A]$}\label{sec_graph_j}\mbox{}

We give the vertices of $\cal G(x)$ shorthand names as in Figure \ref{fig_graph_j}. 

\begin{figure}[h]%
  \resizebox{\textwidth}{!}{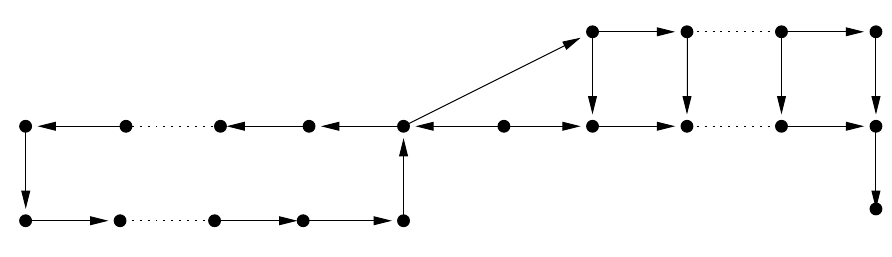}
  \caption{}
  \label{fig_graph_j}
\end{figure}

\begin{lemma}  If $l= 2^{k-1}-1$ then $\dim \ker T(x)= 2$. Otherwise $\dim \ker T(x)=1$.
\end{lemma}

\begin{proof}
We will focus on the case $k>1$. The arguments in the case $k=1$ are very similar and are left to the reader. 

First, assume  $l=2^{k-1}-1$. The first generator of $\ker T(x)$ is the indicator function of the vertex $F$. The coefficients of 
another generator of $\ker T(x)$ are depicted on Figure \ref{fig_graph_j_ker}.

\begin{figure}[h]%
  \resizebox{\textwidth}{!}{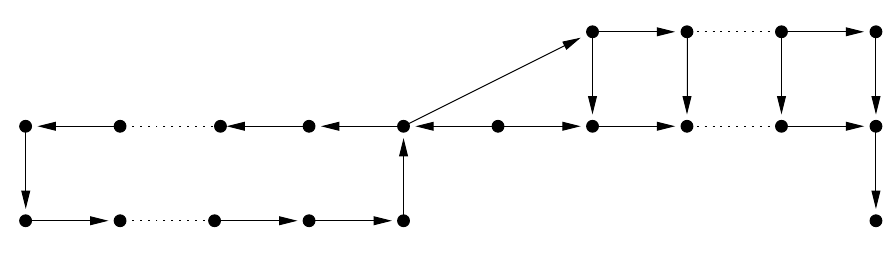}
  \caption{Coefficients of the second generator of $\ker T(x)$ when $l= 2^{k-1}-1$}
  \label{fig_graph_j_ker}
\end{figure}

To see that these two vectors generate all of $\ker T(x)$ let us prove the following.

\begin{lemma*}
 Let $f\in \ker T(x)$ be such that $f(F) = f(A_1)=0$. Then $f=0$.
\end{lemma*}

\begin{proof}
From the flow lemma at $A_2$ we see that $f(A_1)=0$ implies $f(A_2)=0$. Similarly we show $f(A_i) = f(B_i)= 0$ for all $i$. Now the flow lemma at $A_1$ together with $f(A_1) = f(B_k)=0$ implies $f(I)=0$, and the flow lemma at $C_1$ and $f(A_1)=0$ imply $f(C_1)=0$. The flow lemma at $D_1$ together with $f(I) = f(C_1)=0$ implies $f(D_1)=0$.

Now note that the flow lemma at $C_{i+1}$ and $f(C_i) =0$ imply $f(C_{i+1})=0$. Thus we get $f(C_i)= 0$ for all $i$.

Finally the flow lemma at $D_{i+1}$ and $f(D_i) = f(C_{i+1}) =0$ imply $f(D_{i+1}) =0$, and so we also get $f(D_i) =0$ for all $i$. Since $f(F)=0$ by assumption, the claim follows.
\end{proof}

Note that the indicator function of the vertex $F$ is in $\ker T(x)$ for arbitrary $(k,l)$. Thus to finish the proof it is enough to show that if $f\in \ker T$ is such that $f(A_1)= 1$ then $l=2^{k-1}-1$.

So assume $f(A_1)=1$. From the flow lemma at $A_2$ we get $f(A_2)= 2$. Similarly $f(A_i) = 2^{i-1}$ for all $i$, and in particular $f(A_k) =2^{k-1}$. 

Now from the flow lemma at $B_1$ we have also $f(B_1)=2^{k-1}$ and by induction $f(B_k)=2^{k-1}$.

Since $f(A_1) =1$ and $f(B_k) =2^{k-1}$, the flow lemma at $A_1$ implies $f(I) = 2^{k-1}$.  The flow lemma at $C_1$ together with $f(A_1) =1$ implies  $f(C_1)=1$, and by induction $f(C_i)=1$ for all $i's$. Thus by the flow lemma at $D_1$ we get $f(D_1) =2^{k-1}-2$ and inductively $f(D_i) = 2^{k-1}-i-1$.

This means that $f(D_l) =0 $ only if $0= 2^{k-1} -l -1$. Since the flow lemma at $F$ implies $f(D_l)= 0$, this ends the proof.  
\end{proof}



\bibliographystyle{alpha}
\bibliography{bibliografia}

\end{document}